\documentclass[letterpaper,12pt,leqno]{article}

\usepackage{fullpage}

\usepackage{hyperref}

\usepackage{amsmath}
\usepackage{amsfonts}
\usepackage{amssymb}
\usepackage{tabularx}
\usepackage{color}
\usepackage{fancyhdr}
\usepackage{tikz-cd}
\usepackage{enumitem}
\usetikzlibrary{calc}

\def\theoremparent{section}
\usepackage{notes}

\numberwithin{equation}{theorem}

\def\A{\mathbb{A}}

\def\cF{\mathcal{F}}

\def\cL{\mathcal{L}}
\def\cM{\mathcal{M}}

\def\Com{\mathrm{Com}}

\def\d{\partial}

\def\even{{\bar{0}}}

\def\g{\mathfrak{g}}

\def\k{\mathbf{k}}

\def\M{\mathcal{M}}

\def\odd{{\bar{1}}}

\def\rmor{\mathfrak{or}}

\def\proof{\noindent{\em Proof:}\ }

\def\qed{\hfill\lower 1em\hbox{$\square$}\vskip 1em}

\def\tr{\mathrm{tr}}
\def\toto{\text{\ \raise0.2em\hbox to 0pt {$\to$}\lower0.2em\hbox{$\to$}\ }}

\def\ubf#1{\ul{\mathbf{#1}}}

\begin{document}

\title{$L$-manifolds}
\author{Slava Pimenov}
\date{\today}
\titlepage
\maketitle

\tableofcontents
\vskip 5em

\setcounter{section}{-1}
\section{Introduction}
The notion of a Frobenius manifold appears in relation to various topics in algebraic and analytic geometry, such and quantum cohomology, deformation of meromorphic connections,
unfolding of singularities and others. Of a particular interest for us here will be the formal Frobenius manifolds, that is a Frobenius
structure on the formal completion of a manifold at a point. In this situation there are two equivalent descriptions, either as an algebra over a certain cyclic operad $\Com_\infty$,
or alternatively as a genus $0$ cohomological field theory. Let us briefly recall these correspondences, for a detailed discussion of it, as well as general theory of Frobenius
manifolds we refer to \cite{Dubrovin} and \cite{Manin}.

\begin{nparagraph}
Let $X$ be a super-manifold, equipped with a flat structure, i.e. a subsheaf of flat vector fields $T^f_X \into T_X$, as well as an even symmetric non-degenerate form $g\from S^2(T_X) \to \O_X$.
A Frobenius structure on $X$ is, roughly speaking, a commutative and associative binary $\O_X$-linear operation $\circ\from S^2(T_X) \to T_X$, which in flat coordinates is described
by an even potential function $\Phi$, namely for $x, y, z \in T^f_X$ we have $g(x \circ y, z) = xyz(\Phi)$, where on the right side vector fields act as derivations on functions.

If $X$ is a formal super-manifold, then it can be identified with the completion of a super vector space $V$ at zero, equipped with a symmetric non-degenerate pairing
$g\from S^2(V) \to \k$. The potential $\Phi$ is an element of $\widehat S^\bullet (V^*)$, and its homogeneous components then can be thought of as multilinear operations on $V$.
Specifically, we have a collection of symmetric operations $m_n\from S^n(V) \to V$ for $n \ge 2$, such that the maps $Y_{n+1}(v_1, \ldots, v_{n+1}) = g(m_n(v_1, \ldots, v_n), v_{n+1})$
are $S_{n+1}$-invariant, and the potential $\Phi = \sum (Y_n / n!)$, where $Y_n$ is considered as an element of $S^n(V^*)$.
This collection of operations $\{m_n\}$ equips $(V, g)$ with a structure of an algebra over the cyclic operad
$\Com_\infty$.
\end{nparagraph}

\begin{nparagraph}
On the other hand this structure also admits a more algebro-geometric interpretation. Consider the moduli spaces $\wbar \cM_{0,n}$ of stable curves of genus zero with $n$ marked points.
The entire collection $\{\wbar \cM_{0,n}, n \ge 3\}$ forms a cyclic operad with values in algebraic manifolds, where the operadic composition is given by attaching curves along
marked points. Specifically, let $C \in \wbar \cM_{0,n+1}$ with marked points $\{x_0, \ldots, x_n\}$ and similarly $C_i \in \wbar \cM_{0, k_i + 1}$ with marked points
$\{x^i_0, \ldots, x^i_{k_i}\}$ for $1 \le i \le n$, then the map
$$
\wbar \cM_{0, n+1} \times \wbar \cM_{0, k_1 + 1} \times \cdots \times \wbar \cM_{0, k_n + 1} \to \wbar \cM_{0, (1 + \sum k_i)}
$$
at the level of points is given by first taking the disjoint union of curves $C, C_1, \ldots, C_n$ and then identifying points $x_i$ and $x^i_0$ for all $1 \le i \le n$.

At the level of cohomology this induces a structure of a cyclic co-operad consisting of spaces $H^\bullet(\wbar \cM_{0,n})$. A cohomological field theory (of genus $0$) on a space
$(V, g)$ is a co-action of this co-operad. To see the relation of this notion and the previous $\Com_\infty$-algebras let us recall the structure of the cohomology groups
of $\wbar \cM_{0, n}$, which is particularly easy to understand (unlike the case of higher genera) (\cite{Keel}). 

The combinatorial type of a stable curve $C$ of genus $0$ is given by a stable tree, where vertices correspond to the irreducible components of $C$, tails to the marked points
and two vertices are connected by an edge if and only if the corresponding irreducible components of $C$ intersect. The cohomology $H^\bullet(\wbar \cM_{0,n})$ is purely even
and can be read entirely from the stratification of the moduli spaces by the combinatorial type of the stable curve. First, consider the case of $\wbar \cM_{0,4}$, which is
isomorphic to the projective line $\P^1$. The open strata $\cM_{0,4}$ is isomorphic to $\P^1 - \{0, 1, \infty\}$, and its combinatorial type is the corolla tree with a single vertex
and four tails. The three points forming the boundary divisor correspond to the trees with two vertices, one internal edge, and the tails labeled in one of three possible ways
$(1,2 \mid 3, 4)$, $(1, 3 \mid 2, 4)$ and $(1, 4 \mid 2, 3)$.
$$
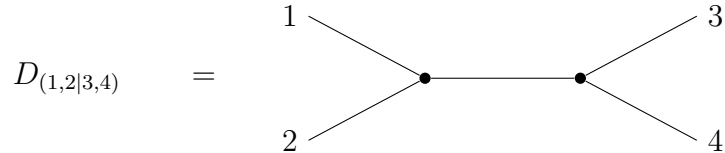
\begin{tikzpicture}
\def\u{5em}
\node at (-2*\u, 0) {$D_{(1, 2 \mid 3, 4)} \quad\quad = $};

\node (v0) at (0,0) [circle,fill,inner sep=1.5pt] {};
\node (v1) at (1*\u,0) [circle,fill,inner sep=1.5pt] {};

\draw [-] (v0) -- (v1);
\draw [-] (v0) -- (-0.75*\u,0.4*\u) node[at end,left] {$1$};
\draw [-] (v0) -- (-0.75*\u,-0.4*\u) node[at end,left] {$2$};
\draw [-] (v1) -- (1.75*\u,0.4*\u) node[at end,right] {$3$};
\draw [-] (v1) -- (1.75*\u,-0.4*\u) node[at end,right] {$4$};
\end{tikzpicture}
$$

In the cohomology ring $H^\bullet(\P^1)$ (and in fact in the Chow ring $A(\P^1)$) the classes of divisors corresponding to these three points coincide, i.e.
\begin{equation}
\label{eq_fund_keel}
D_{(1, 2 \mid 3, 4)} = D_{(1, 3 \mid 2, 4)} = D_{(1, 4 \mid 2, 3)}.
\end{equation}
Now, interpreting tails $(1, 2, 3)$ as inputs, tail $4$ as the output and assigning to each
vertex the binary operation $m_2$ it is clear that these relations express the associativity of $m_2$.

Now the relations (\ref{eq_fund_keel}) are in a sense fundamental, and imply all other relations in the cohomology $H^\bullet(\wbar \cM_{0,n})$ for all $n$. Consider
the stable forgetful maps $p_{ijkl}\from \wbar \cM_{0,n} \to \wbar \cM_{0,4}$, that forget all marked points except four points marked by $i, j, k, l$ and then stabilize the resulting
curve. The preimage of the strata $D_{(1,2\mid 3,4)}$ is the closure of the union of strata corresponding to trees with two vertices, one internal edge and the tails attached to
the vertices in all possible ways, such that tails $i,j$ are on one side and tails $k,l$ are on the other. Keel has shown that the cohomology $H^\bullet(\wbar \cM_{0,n})$ as an algebra
is generated by classes of strata corresponding to all two-vertex trees modulo the preimage of relations (\ref{eq_fund_keel}) along $p_{ijkl}$ for all quadruples $\{i, j, k, l\}$.
Interpreting these relations from the operadic point of view by assigning to a vertex with $n+1$ attached flags the operation $m_n$ we recover the higher associativity relations
in the operad $\Com_\infty$, thus establishing equivalence of the structures of a cyclic $\Com_\infty$-algebra and a cohomological field theory.
\end{nparagraph}

\begin{nparagraph}
In his book on Frobenius manifolds (\cite{Manin}) Yuri Manin asks what would be the analog of this picture if one starts with the $\Lie_\infty$-operad instead of the $\Com_\infty$.
In this paper we outline the first steps towards answering this question. First, in section \ref{sec_main_def} we give a definition of an $L$-manifold, which is the global geometric notion extending $\Lie_\infty$-algebras,
similar to how Frobenius manifold is the global version of the $\Com_\infty$-algebra. Here we have to work with the {\em curved} cyclic $\Lie_\infty$ algebras in order to have the
same pattern as in the Frobenius setting. Then in section \ref{sec_examples} we give several basic examples of $L$-manifolds, including explicit description in the cases of low odd dimension.
Then in sections \ref{sec_coh_FT} and \ref{sec_graph_com} we define spaces $H^\bullet_S$ which we use to define a version of cohomological field theory equivalent to the notion of
a curved cyclic $\Lie_\infty$-algebra. 

Of a particular interest here is that combinatorially all the constructions go parallel to the case of Frobenius manifolds and moduli spaces of stable curves. This leads us to
formulate two questions.
\end{nparagraph}

\begin{conjecture}
There exist a family of spaces (or rather Artin stacks) $\cL_{0,n}$ forming a cyclic operad, and equipped with the ``forgetful'' maps
$p\from \cL_{0,n} \to \cL_{0,m}$, for $m \le n$, as well as the ``section'' maps $x_i\from \cL_{0,n} \to \cL_{0,n+1}$ for $1 \le i \le n+1$, satisfying the obvious relations (as for $\wbar \cM_{0,n}$),
such that the spaces $H^\bullet_{\ubf{n}} \isom H^\bullet(\cL_{0,n})$.
\end{conjecture}

\begin{question}
If such spaces $\cL_{0,n}$ exist, are they moduli spaces of something? And if the answer is positive, is there a natural higher genus version of these space $\cL_{g,n}$, and what is
the corresponding higher genus notion of $\Lie_\infty$-algebra?
\end{question}

\vskip 1em
The author would like to thank BIMSA (Beijing Institute of Mathematical Sciences and Applications) for providing excellent working conditions during preparation of this paper.

\vskip 5em
\section{Main definitions}
\label{sec_main_def}

A super-manifold $X$ is a pair $(X_0, \O_X)$, where $X_0 = (X_0, \O_{X_0})$ is a classical smooth manifold in one of the usual categories, like $C^\infty$, complex-analytic, algebraic manifolds, etc.
And $\O_X = (\O_X)_\even \oplus (\O_X)_\odd$ is a $\Z/2$-graded sheaf of algebras over $X_0$, satisfying the following conditions.
\begin{itemize}
\item The quotient of $\O_X$ by the ideal generated by odd elements $\O_X / ((\O_X)_\odd)$ is isomorphic to $\O_{X_0}$.
\item Locally on $X_0$ (in appropriate topology) $\O_X \isom \O_{X_0} \tensor \Lambda^\bullet(V)$ for some free $\O_{X_0}$-module $V$, such that the generating elements from $V$ are placed in odd degree.
\end{itemize}

We denote $T_X$ the tangent bundle on $X$, in other words the $\Z/2$-graded sheaf of even and odd self-derivations of $\O_X$. And consider an even (super-)antisymmetric bilinear
form $\omega\from \Lambda^2(T_X) \to \O_X$, which induces an isomorphism $\omega'\from T_X \to T_X^*$. If we denote by $\Pi T_X$ the parity shifted tangent bundle, then giving
form $\omega$ is equivalent to giving an even symmetric bilinear form $g\from S^2(\Pi T_X) \to \O_X$, inducing an isomorphism $g' \from \Pi T_X \to \Pi T_X^*$.

A flat affine structure on a super-manifold $X$ is a subsheaf $T_X^f \into T_X$, such that it locally spans the tangent bundle, i.e. $T_X \isom T_X^f \tensor \O_X$, and for any two sections
$v, w \in T_X^f$, called flat vector fields, the Lie bracket $[v, w] = 0$. Notice that in the $C^\infty$ and complex-analytic settings we can find local coordinates $x_i$ on $X$, such that
the partial derivatives $\d_i = \d / \d x_i$ form a local basis of $T_X^f$. If we choose a covering of $X$ by charts equipped with such local coordinates, the transition maps between
charts are given by affine linear transformations.

If the bilinear form $\omega$ is compatible with the flat affine structure on $X$, then it is necessarily closed, and therefore is a symplectic form on $X$.

\begin{definition}
\label{def_lman}
\begin{enumerate}[label=\alph*)]
\item A pre-$L$-manifold is a symplectic super-manifold $(X, \omega)$ equipped with the affine flat structure $T_X^f$ compatible with the symplectic form, and an odd vector field $Q \in \Gamma(X, \Pi T_X)$.
\item
\label{def_potentiality}
A local potential is a local odd function $\Phi$ on $X$, such that the odd 1-form $\sigma = \omega'(Q) = d\Phi$. In other words $Q$ is an odd Hamiltonian vector field with Hamiltonian $\Phi$:
$$
\omega(X, Q) = X \Phi.
$$
A pre-$L$-manifold is said to be potential if it admits a global potential $\Phi$.
\item A potential pre-$L$-manifold is called an $L$-manifold if it satisfies the Lie condition (as will be explained later)
\begin{equation}
\label{eq_lie_cond}
d (\omega(Q, Q)) = 0.
\end{equation}
\end{enumerate}
\end{definition}

Let $(V, \omega)$ be a finite dimensional symplectic super vector space, then a {\em formal} $L$-manifold structure on $(V, \omega)$ is an $L$-manifold, with the underlying
space $X$ being the completion of $V$ at zero, so that the algebra of functions
$$
\O_X = \widehat{S^\bullet}(V^*) \isom \k[[x_1, \ldots, x_m]] \tensor \Lambda[\xi_1, \ldots, \xi_n],
$$
and the symplectic structure is induced by $\omega$ after identifying the tangent sheaf with $\O_X \tensor V$.

\begin{proposition}
If $(X, \omega, Q)$ is an $L$-manifold, then the super-commutator $[Q, Q] = 0$.
\end{proposition}
\proof
Denote by $\{-,-\}$ the Poisson bracket on $\O_X$ induced by the symplectic form $\omega$. Then for any function $f \in \O_X$ we have
$Q f = \{\Phi, f\}$ and therefore
$$
[Q, Q] f = 2 Q^2 f = 2\{\Phi, \{\Phi, f\}\} = \{\{\Phi, \Phi\}, f\}.
$$
On the other hand $\{\Phi, \Phi\} = Q \Phi = \omega(Q, Q)$, which according to the Lie condition (\ref{eq_lie_cond}) is a constant function. Hence the Poisson bracket
on the right hand side vanishes and we conclude that $[Q, Q] = 0$.
\qed

\begin{nparagraph}[Differential equation for $\Phi$.]
Consider a system of local flat coordinates $\{x_i\}$ and let $\{\d_i = \d / \d x_i\}$ be the corresponding basis of flat vector fields, determined by condition that
$$
df = \sum dx_i (\d_i f).
$$
To simplify notation we will also write for iterated derivatives $f_{i_1 \ldots i_k} := \d_{i_1} \ldots \d_{i_k} f$.
We denote $\omega_{ij} = \omega(\d_i, \d_j)$ and let $\omega^{ij}$ be the inverse matrix. These matrices satisfy the following symmetry relations
$$
\omega_{ij} = -(-1)^{\bar x_i \bar x_j} \omega_{ji},\quad\quad \omega^{ij} = (-1)^{(\bar x_i + 1)(\bar x_j + 1)} \omega^{ji}.
$$
Since by assumption form $\omega$ is compatible with the flat structure, the coefficients of these matrices are elements of the ground ring.
If the latter is purely even (for instance if we are working over the field $\k$) then $\omega_{ij}$ is non-zero only if $\bar x_i = \bar x_j$,
in which case the symmetry condition for $\omega^{ij}$ coincides with that for $\omega_{ij}$. To simplify the signs we will assume that we are
working over a purely even base ring.

Using this notation the potentiality condition \ref{def_potentiality} reads
$$
Q = \sum \omega^{ji} \Phi_i \d_j.
$$
And the condition (\ref{eq_lie_cond}) is equivalent to the system of non-linear differential equations for the potential $\Phi$:
\begin{equation}
\label{eq_diff_eq}
\sum_{i, j} \omega^{ji} \Phi_{ki} \Phi_j = 0.
\end{equation}
\end{nparagraph}

\begin{nparagraph}[Curved cyclic $\Lie_\infty$-algebra.]
Let $V = V_\even \oplus V_\odd$ be a finite dimensional super vector space, equipped with an even symmetric non-degenerate bilinear form $g\from S^2(V) \to \k$. Consider a collection of
antisymmetric multilinear operations $\ell_n\from \Lambda^n(V) \to \Pi^n V$, for $n \ge 0$, where $\Pi$ stands for the parity change operator. In other words
operation $\ell_n$ has arity $n$ and parity $(n\ \mathrm{mod}\ 2)$. In particular $\ell_0$ is an element of $V_\even$, which will be referred to as the {\em curvature},
and $\ell_1$ is an odd endomorphism of $V$, which will be called a {\em differential}. Let us also denote by $\ubf{n}$ the set $\{1, 2, \ldots, n\}$.
\end{nparagraph}

\begin{definition}
\label{def_ccla}
A curved cyclic $\Lie_\infty$-algebra structure on $V$, is a collection $\{\ell_n\}$, satisfying the following conditions.
\begin{enumerate}[label=\alph*)]
\item Cyclic invariance: operations $Y_{n+1} \from V^{\tensor (n+1)} \to \k$ defined as $Y_{n+1}(x_1, \ldots, x_{n+1}) = g(\ell_n(x_1, \ldots, x_n), x_{n+1})$ are $\Z/(n+1)$-invariant (in the super sense), i.e.
$$
Y_{n+1}(x_1, \ldots, x_n, x_{n+1}) = (-1)^{n + \bar x_1 (\bar x_2 + \cdots + \bar x_{n+1})} Y_{n+1}(x_2, \ldots, x_{n+1}, x_1).
$$
\item Higher Jacobi relations: for all $n \ge 0$ we have
\begin{equation}
\label{eq_Jacobi}
\sum_{\ubf{n} = S_1 \sqcup S_2} (-1)^{|S_1||S_2|}\epsilon(S_1, S_2) \ell_{|S_2| + 1}(\ell_{|S_1|}(x_i \mid i \in S_1), x_j \mid j \in S_2) = 0.
\end{equation}
\end{enumerate}
\end{definition}
Here the sum is taken over all partitions of $\ubf{n}$ into two subsets (possibly empty), and $\epsilon(S_1, S_2)$ is the sign of the permutation of inputs $\{x_i\}$,
by the shuffle of $\ubf{n}$ obtained by first arranging all the elements of $S_1$ in the increasing order followed by all the
elements of $S_2$, also in the increasing order. Notice, that the sign here also involves the degrees of the inputs $x_i$, for instance, sign
of the transposition of the two adjacent elements $(i, i+1)$ is $(-1)^{1+\bar x_i \bar x_{i+1}}$.

For convenience we will write $c := \ell_0$ for the curvature element, $d := \ell_1$ for the differential and $[x, y] = \ell_2(x, y)$ for the binary operation.
The first few identities (\ref{eq_Jacobi}) for small $n$ are as follows:
\begin{enumerate}[label=\alph*)]
\item for $n = 0$ we have $d(c) = 0$,
\item for $n = 1$ we have $d^2(x) + [c, x] = 0$,
\item for $n = 2$ we have $d[x, y] - [dx, y] - (-1)^{\bar x} [x, dy] + \ell_3(c, x, y) = 0$,
\item for $n = 3$ we have
\begin{multline*}
d \ell_3(x, y, z) + [[x, y], z] - (-1)^{\bar y \bar z}[[x, z], y] + (-1)^{\bar x(\bar y + \bar z)}[[y, z], x] + \\
\ell_3(dx, y, z) + (-1)^{\bar x} \ell_3(x, dy, z) + (-1)^{\bar x + \bar y} \ell_3(x, y, dz) + \ell_4(c, x, y, z) = 0.
\end{multline*}
\end{enumerate}

In particular, if we forget the pairing $g$ on $V$ and put curvature $c = 0$, we recover the standard notion of a homotopy Lie algebra.

\begin{proposition}
\label{prop_lman_lie}
Let $(V, \omega)$ be a finite dimensional symplectic super vector space. Then there is a bijection between the set of formal $L$-manifold structures on $(V, \omega)$
and the set of curved cyclic $\Lie_\infty$-structures on $(\Pi V, g = \omega^{\Pi})$. It sends a collection of operations $\{\ell_n\}$ to the potential

$$
\Phi = \sum_{n \ge 1} {Y_n \over n!},
$$
where $Y_n$ is considered as an element of $\Lambda^n(\Pi V^*) \isom S^n(V^*)$.
\end{proposition}
\proof
Let $\{x_i\}$ be a system of flat coordinates, and as before denote by $\{\d_i\}$ the corresponding basis of flat vector fields. Furthermore, we identify the space $\Pi V$
with the fiber over zero of the parity shifted tangent bundle, and denote by $e_i = \d_i^\Pi|_0$ the corresponding basis vectors of $\Pi V$.

For any super vector space $W$ we will fix an explicit isomorphism $\Lambda^n(W) \isom S^n(\Pi W)$ by sending
$$
w_1 \wedge w_2 \wedge \ldots \wedge w_n \mapsto (-1)^{\sum_{i = 1}^n (i - 1) \bar w_i} w_1^\Pi w_2^\Pi \ldots w_n^\Pi.
$$
Let us denote the sign on the right hand side by $\lambda(w_1, \ldots, w_n)$.

Looking at the linear terms of $\Phi$ we find that $(c, e_i) = \Phi_i(0)$, and on the other hand if we write $c = \sum C^j e_j$ we have
$$
(C^j e_j, e_i) = (-1)^{\bar x_i + 1} \omega_{ji}.
$$
Therefore $C^j = \sum \Phi_i(0) \omega^{ij}$, since $\Phi_i(0)$ can only be non-zero if $\bar x_i = 1$.
Similarly, if we write $D_i^k$ for the matrix coefficients of the differential $d = \ell_1$, i.e.
$$
d(e_i) = \sum_k D_i^k e_k,
$$
then by looking at the quadratic terms of $\Phi$ we find on one hand
$$
(d(e_i), e_j) = (-1)^{\bar x_j + 1} \Phi_{ij}(0),
$$
and on the other hand
$$
(d(e_i), e_j) = \sum_k (-1)^{\bar x_j + 1} D_i^k  \omega_{kj}.
$$
Therefore, multiplying on the right side by $\omega^{jk}$ we see that $D_i^k = \Phi_{ij}(0) \omega^{jk}$. And so the first of relations (\ref{eq_Jacobi}) $d(c) = 0$ is
equivalent to vanishing of the system (\ref{eq_diff_eq}) at the origin.

We will now show that the rest of the higher Jacobi conditions are equivalent to vanishing of the iterated derivatives of the differential equations (\ref{eq_diff_eq}) at the origin.
Therefore, since $\Phi$ is a formal series, this will imply that the equations hold.

It may be helpful to represent the formulas here graphically, using planar trees.

\begin{nparagraph}
\label{par_ribbon_graph}
Recall that a ribbon graph is a collection $(V, F, \delta, \sigma, j)$, where $V$ is the set of vertices, $F$ is the set of flags (or half-edges), $\delta\from F \to V$ is the incidence
map which attaches a flag to a specific vertex, $\sigma$ is a collection of cyclic orders $\sigma_v$ on the set $F_v := \delta^{-1}(v)$ for each vertex $v \in V$, and $j \from F \to F$ is
an involution on the set of flags which prescribes gluing of pairs of flags into edges. A planar tree is a ribbon graph without cycles.

Let $T$ be a connected planar tree, a tail of $T$ is a flag fixed by the involution $j$. An orientation of $T$ is a map $or\from F \to \{\mathrm{in}, \mathrm{out}\}$, such that all two-element
orbits of $j$ map surjectively onto $\{\mathrm{in}, \mathrm{out}\}$. If a tree has at least one tail, we will choose one of them and orient it outward, which in turn
induces an orientation on the entire tree, by requiring that every set $F_v$ has exactly one outward flag. Denote by $\mathrm{Tail}(T)$ the set of tails of $T$ and by $E(T)$ the set of
internal edges of $T$.

We will decorate the tree by assigning to every flag $f \in F$ an element $m(f)$ in the basis of flat vector fields. More precisely, if it is a tail we will decorate it
by $dx_{m(f)}$, and if it is an internal edge $(i \to j)$ we decorate it by $\omega^{ij}$. By assigning to every vertex of $T$ a function we obtain an element of
$(T^*_V)^{\tensor \mathrm{Tail}(T)}$, by reading inputs, say in a counterclockwise order, starting from the outward tail.

For example the following tree
$$
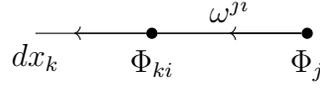
\begin{tikzpicture}
\def\u{5em}
\node (v0) at (0,0) [circle,fill,inner sep=1.5pt,label=below:$\Phi_{ki}$] {};
\node (v1) at (1*\u,0) [circle,fill,inner sep=1.5pt,label=below:$\Phi_{j}$] {};

\draw [-] (v1) -- (v0);
\draw [->] (v1) -- (0.5*\u,0) node[at end,above] {$\omega^{ji}$};
\draw [-] (v0) -- (-0.75*\u,0) node[at end,below] {$d x_k$};
\draw [->] (v1) -- (-0.5*\u,0);
\end{tikzpicture}
$$
represents the form: $\sum d x_k (\omega^{ji} \Phi_{ki} \Phi_j)$.

Differentiating this expression with respect to $x_m$ corresponds to adding an additional inward flag to either one of the vertices and then taking sum of the
resulting trees. Thus in this example we obtain

\begin{equation}
\label{eq_pic_d2}
\begin{tikzpicture}[baseline=-0.25em]
\def\u{5em}
\node (v0) at (0,0) [circle,fill,inner sep=1.5pt,label=below:$\Phi_{mki}$] {};
\node (v1) at (1*\u,0) [circle,fill,inner sep=1.5pt,label=below:$\Phi_{j}$] {};

\draw [-] (v1) -- (v0);
\draw [->] (v1) -- (0.5*\u,0) node[at end,above] {$\omega^{ji}$};
\draw [-] (v0) -- (-0.75*\u,0) node[at end,below] {$d x_k$};
\draw [->] (v1) -- (-0.5*\u,0);
\draw [-] (0, 0.75*\u) -- (v0) node[at start,left] {$d x_m$};
\draw [->] (0, 0.75*\u) --(0, 0.5*\u);

\node at (2*\u,0) {$+(-1)^{\bar x_m(1 + \bar x_k)}$};
\def\v{3.6*\u}
\node (v2) at (0+\v,0) [circle,fill,inner sep=1.5pt,label=below:$\Phi_{ki}$] {};
\node (v3) at (1*\u+\v,0) [circle,fill,inner sep=1.5pt,label=below:$\Phi_{jm}$] {};

\draw [-] (v3) -- (v2);
\draw [->] (v3) -- (0.5*\u+\v,0) node[at end,above] {$\omega^{ji}$};
\draw [-] (v2) -- (-0.75*\u+\v,0) node[at end,below] {$d x_k$};
\draw [->] (v3) -- (-0.5*\u+\v,0);
\draw [-] (1.75*\u+\v, 0) -- (v2) node[at start,below] {$d x_m$};
\draw [->] (1.75*\u+\v, 0) --(1.5*\u+\v, 0);
\end{tikzpicture}
\end{equation}

Arguing as before, if we write $\ell_n(e_{i_1}, \ldots, e_{i_n}) = B_{i_1 \ldots i_n}^k e_k$, we have
\begin{align*}
(\ell_n(e_{i_1}, \ldots, e_{i_n}), e_j) &= \lambda(e_{i_1}, \ldots, e_{i_n}, e_j) \Phi_{i_1 \ldots i_n, j}(0) \\
&= (-1)^{\bar e_j} B_{i_1 \ldots i_n}^k \omega_{kj}.
\end{align*}

Let us evaluate form in (\ref{eq_pic_d2}) at zero. Notice that for non-zero terms we can cyclically shift indices of derivatives of $\Phi$ without changing sign.
So for the first tree we have $\Phi_{imk}(0) = (-1)^{\bar x_m + \bar x_k} B_{im}^l \omega_{lk}$ and since $\bar x_m = \bar x_k$, multiplying by $\omega^{kl} e_l$
we obtain the commutator $[c,e_m]$. Similarly for the second tree we obtain $d^2(e_m)$, which together gives the second of the Jacobi identities (\ref{eq_Jacobi}).

Iterating this process of adding more new tails labeled by elements of a set $S$, it is clear that the resulting trees will be indexed by partitions of $S$ into two disjoint
subsets $S_1 \sqcup S_2$ and by direct examination we see that we end up with all the relations in (\ref{eq_Jacobi}).
\qed

\end{nparagraph}

\begin{remark}
Alternatively, the bijection between formal $L$-manifolds and curved cyclic $\Lie_\infty$-algebras can be described as follows. Let $V$ be a formal $L$-manifold with
a potential $\Phi$, then the odd 1-form $d\Phi$ is a map $T_V \to \Pi \O_V$, which is completely determined by its value on the flat vector fields, that is a map
$V \to \Pi \O_V \isom \Pi S^\bullet(V^*)$. Using isomorphism between $V$ and $V^*$ induced by the symplectic form we obtain
$$
V^* \to \Pi S^\bullet(V^*),
$$
which in turn induces an odd derivation of the symmetric algebra. The Lie condition (\ref{eq_lie_cond}) ensures that this map is square-zero, and gives
an alternative definition of the $\Lie_\infty$-algebra structure on $\Pi V$ as a differential in the Chevalley-Eilenberg complex.
\end{remark}

\begin{definition}
We say that an $L$-manifold $(X, \omega, Q)$ is isotropic if $\omega(Q, Q) = 0$.
\end{definition}

\begin{nparagraph}[Euler vector field.]
\label{par_euler}
Let $E$ be an even vector field on an $L$-manifold $X$, which is conformal, that is $\Lie_E (\omega) = D \omega$ for some constant $D$, or equivalently
$$
E (\omega(V, W)) - \omega([E, V], W) - \omega(V, [E, W]) = D \omega(V, W).
$$
We say that a conformal vector field $E$ is Euler if in addition it preserves the flat structure and satisfies condition
$$
[Q, E] = d_Q Q
$$
for some constant $d_Q$.

Since by definition of an $L$-manifold the pairing $\omega(Q, Q)$ is a constant function we have
$$
0 = E (\omega(Q, Q)) = D \omega(Q, Q) + \omega([E, Q], Q) + \omega(Q, [E, Q]) = (D - 2d_Q) \omega(Q, Q).
$$
In other words an $L$-manifold $X$ with an Euler vector field $E$ has to be isotropic unless $D = 2d_Q$.

In terms of the potential $\Phi$ the Euler vector field conditions imply that up to a constant $\Phi$ is an eigenvector of $E$:
\begin{equation}
\label{eq_euler_phi}
E \Phi = (D - d_Q) \Phi + \mathrm{const}.
\end{equation}
Indeed, differentiating along a vector field $X$, we have $X(E\Phi) = E(X \Phi) - [E, X] \Phi$. On the other hand, by definition of potential $X \Phi = \omega(X, Q)$,
and hence
\begin{multline*}
X(E \Phi) = E(\omega(X, Q)) - \omega([E, X], Q) =\\
 D \omega(X, Q) + \omega([E, X], Q) + \omega(X, [E, Q]) - \omega([E, X], Q) = (D - d_Q) \omega(X, Q),
\end{multline*}
which implies (\ref{eq_euler_phi}).

\end{nparagraph}

\vskip 5em
\section{Examples}
\label{sec_examples}
We begin with two extreme examples. First, if $X$ is a purely even manifold, then the ring of functions $\O_X$ is purely even and the potential
$\Phi$, being an odd function, has to be zero. In other words a purely even manifold doesn't admit a non-trivial $L$-manifold structure.

On the other hand, if $X$ is a purely odd manifold, then it is formal, and its ring of function is a finite dimensional exterior algebra in odd variables. Hence,
by proposition \ref{prop_lman_lie}, it is given by a purely even curved cyclic $\Lie_\infty$-algebra $\g$ with only even operations $\ell_n$, $n \ge 0$.

\begin{nparagraph}[Flat vector field $Q$.]
\label{par_flat_Q}
If $X$ is a pre-$L$-manifold, such that odd vector field $Q$ is a flat, then both the potentiality condition and
$\omega(Q, Q) = \mathrm{const}$ are automatically satisfied. Specifically, in this case the potential is given by a linear function in odd flat coordinates
$\Phi = \sum a_i \xi_i$. In terms of the $\Lie_\infty$-algebra structure, the shifted tangent bundle $\Pi T_X$ is equipped with a locally constant even section $c$, i.e. the curvature,
and the rest of the $\Lie_\infty$ operations are trivial.

\end{nparagraph}

\begin{nparagraph}[Lie super-algebra.]
\label{par_lie_alg}
For any Lie super-algebra $\g$ of dimension $(m, n)$ equipped with a non-degenerate symmetric invariant bilinear form we can construct an $L$-manifold with
the underlying superspace $X = \A^{n,m} \isom \Pi \g$, by proceeding as in proposition \ref{prop_lman_lie}. Since all the operations $\ell_n$ with $n \ge 3$
vanish, the $L$-manifold structure in the formal neighborhood of zero extends to the entire affine space. The potential in the coordinates dual to $\g$
is given by an odd cubic (homogeneous) polynomial. Of course at other points of $X$ the fibers of $\Pi T_X$ in general may have nonzero curvature and differential,
but no higher operations $\ell_n$ with $n \ge 3$.

Let us work out in more detail the case of super-algebra $\g = \gl(m,n) = \gl(V)$. Let $\{v_i\}$ be a basis of $V$, and write $e_{ij} = v_i \tensor v_j^*$ for
the corresponding basis of $\gl(m, n)$. And consider the bilinear form
$$
(e_{ij}, e_{ji}) = (-1)^{\bar v_i},
$$
which satisfies all of the above assumptions. Denote by $x_{ji}$ the flat coordinate on $X$ dual to basis element $e_{ij}$, in other words we have $\d/\d x_{ji} = \d_{ij} = e_{ij}^\Pi$.
Then the potential is given by
$$
\Phi = \sum_{(i, j, k)} (-1)^{\bar v_i + \bar v_j + \bar v_k} x_{ik} x_{kj} x_{ji}.
$$
Here the sum is taken over all triples $(i, j, k)$ up to cyclic shift. Notice that since the degree $\bar x_{ii} = 1$, the terms with $i = j = k$ do not contribute to the sum.
\end{nparagraph}

\begin{proposition}
The $L$-manifold $X$ associated to a Lie super-algebra $\g$ is isotropic.
\end{proposition}
\proof
Let $\{e_i\}$ be a basis of $\g$ and denote by $\{x_i\}$ the corresponding system of flat coordinates on $X$, such that $\d_i = \d/\d x_i = e_i^\Pi$. Consider an even vector field
$$
E = \sum_i x_i \d_i.
$$
Clearly, we have $E x_i = x_i$ and $[\d_i, E] = \d_i$ for all $i$, and therefore $E$ is an Euler vector field with $D = 2$ and $d_Q = -1$. Since $D \neq 2 d_Q$ this implies,
as was shown in paragraph \ref{par_euler}, that $\omega(Q, Q) = 0$.
\qed

\begin{remark}
Similar argument can be used to show that if $\g$ is a cyclic curved $\Lie_\infty$ algebra, such that all operations $\ell_n$ vanish except for a single $n > 0$, then
$E$ as above is again an Euler vector field on the associated $L$-manifold $X$, with $D = 2$ and $d_Q = (1 - n)$. Therefore it is once again isotropic.

It would be interesting to characterize isotropic (formal) $L$-manifolds in terms of the corresponding $\Lie_\infty$-algebra structures.
\end{remark}

\begin{nparagraph}[$L$-manifolds of dimension $(n, 1)$.]
Consider a super-manifold $X$ of dimension $(n, 1)$, and denote by $(x_1, \ldots, x_n)$ a system of flat even coordinates on $X$ and by $\xi$ a flat odd coordinate. By rescaling $\xi$
if necessary we may assume that $\omega(\d_\xi, \d_\xi) = 1$. Any odd function $\Phi$ can be written as $\Phi = \phi(x_1, \ldots, x_n) \xi$, therefore all of the derivatives
$\Phi_{x_i}$ have factor $\xi$. The Lie condition (\ref{eq_lie_cond}) can be written as
$$
\phi(x_1, \ldots, x_n)^2 + \sum_{i,j} \omega^{ji} \Phi_{x_i} \Phi_{x_j} = \mathrm{const}.
$$
The sum in this expression contains factor $\xi^2$ and therefore vanishes, which implies that $\phi$ is a constant function. Hence in this case the vector field $Q$ is flat and we
are in the situation discussed in paragraph \ref{par_flat_Q}.
\end{nparagraph}

\begin{nparagraph}[$L$-manifolds of dimension $(n, 2)$.]
Consider a super-manifold $X$ of dimension $(n, 2)$ and, as before, let $(x_1, \ldots, x_n)$ be a system of flat even coordinates on $X$, and $(\xi_1, \xi_2)$ a pair of flat odd coordinates.
By applying a linear transformation with constant coefficients to the odd coordinates we can assume that the restriction of the form $\omega$ to flat odd vector fields is given by matrix
$$
\omega\mid_{\<\d_{\xi_1}, \d_{\xi_2} \>} \ =\ \begin{pmatrix}
0 & 1 \\
1 & 0
\end{pmatrix}.
$$
We will also denote by $\wbar\omega$ the restriction of form $\omega$ to the even manifold $X_0$, and let $\{-,-\}$ be the Poisson bracket corresponding to $\wbar\omega$.

The potential $\Phi$ can be written as
$$
\Phi = \phi_1(x_1, \ldots, x_n) \xi_1 + \phi_2(x_1, \ldots, x_n) \xi_2,
$$
for some even functions $\phi_1$ and $\phi_2$. By separating derivatives with respect to even and odd variables, the Lie condition $\omega(Q, Q) = C$ for some constant $C$
is equivalent to the following two conditions.

$$
\phi_1 \phi_2 = C,
$$
$$
\{ \phi_1, \phi_2 \} = 0.
$$

If $C = 0$ then one of the factors $\phi_1, \phi_2$ is zero (at least locally if we are working in the $C^\infty$ setting) and the second condition is obviously satisfied.
Therefore the potential has the form $\Phi = \phi \xi_1$, for an arbitrary even function $\phi$.

If $C \neq 0$ then $\phi_2 = C / \phi_1$ and again, since $\phi_1, \phi_2$ are functionally dependent, the second condition is satisfied. Therefore we have
$$
\Phi = \phi \xi_1 + {C \xi_2 \over \phi}
$$
for an arbitrary even function $\phi$.

\end{nparagraph}

\vskip 5em
\section{Cohomological field theory}
\label{sec_coh_FT}

In this section we will give a description of the curved cyclic $\Lie_\infty$-algebras in the flavor of cohomological field theory. We define graded algebras $H^\bullet_S$ for any
finite set $S$ combinatorially, and show that a coaction of this collection $\{H^\bullet_S\}$ on a super vector space $V$ is equivalent to giving a structure of a curved cyclic
$\Lie_\infty$-algebra.

It would be interesting to describe the spaces $H^\bullet_S$ as the cohomology of some moduli spaces (or more likely Artin stacks), similar to the relation between the formal
Frobenius manifolds and cohomological field theories with spaces of (co)operations being the cohomology of the moduli of stable curves $\M_{0,S}$.

\begin{nparagraph}
Let $S$ be a finite set (possibly empty) and $\sigma$ be a partition of $S$ into two disjoint subsets $S_1$ and $S_2$. Such partitions are in bijection with trees with two
vertices $v_1$ and $v_2$, one internal edge and tails labeled by elements of $S$. Specifically, given such a tree we put $S_i$ to be the set of tails attached to vertex $v_i$, for $i = 1, 2$.

For any tree $\tau$ with tails labeled by $S$ we consider a symbol $m(\tau)$ of degree $|E(\tau)|$.
Fix a vertex $v \in \tau$ with at least one adjacent flag, and a flag $f \in F_v$, and let $F_v^\circ = F_v - \{f\}$ be the set of all the remaining adjacent flags.
Then for every partition of $F_v^\circ = F_1 \sqcup F_2$ (either $F_1$ or $F_2$ can be empty)
we can construct a new tree $\tau' = \tau'(\tau, v, f, F_1, F_2)$
obtained from $\tau$ by adding a new vertex $v'$, connecting it to $v$ by an edge $e$ and attaching all flags from $F_1$ to $v'$ and all flags from $F_2$ to $v$.

For any such triple $(\tau, v \in \tau, f \in F_v)$ we can consider a linear combination
\begin{equation}
\label{eq_relations}
R(\tau, v, f) := \sum_{F_v^\circ = F_1 \sqcup F_2} (-1)^{|F_1||F_2|} \epsilon(F_1, F_2) m(\tau'(\tau, v, f, F_1, F_2)).
\end{equation}
\end{nparagraph}

\begin{definition}
The space $H^{n}_S$ is the quotient of the linear span of symbols $m(\tau)$ for all trees $\tau$ with $n$ internal edges and tails labeled by $S$, modulo subspace spanned by elements $R(\tau', v, f)$
for all trees $\tau'$ with $(n - 1)$ internal edges, vertices $v \in \tau'$ and flags $f \in F_v$.
\end{definition}

For any partition $\sigma$ of the set of tails $S = S_1 \sqcup S_2$ we can define a boundary map $\phi_\sigma\from H^\bullet_S \to H^\bullet_{S_1^+} \tensor H^\bullet_{S_2^+}$, where
$S_i^+$ denotes the disjoint union of $S_i$ and one additional element. Explicitly, if a tree $\tau$ doesn't contain an edge which induces partition of tails into subsets $S_1$ and $S_2$
then $\phi_\sigma(m(\tau)) = 0$. Otherwise, we cut $\tau$ along this edge into two trees $\tau'$ and $\tau''$ and put $\phi_\sigma(m(\tau)) = m(\tau') \tensor m(\tau'')$.

\begin{nparagraph}
For a vector space $V$ we can consider a cohomological field theory on $V$ consisting of a collection of even maps
$$
I_n \from (\Pi V)^{\tensor n} \to \Pi H^\bullet_n,
$$
which satisfy the following two conditions.
\begin{enumerate}[label=\alph*)]
\item $I_n$ is $S_n$ invariant, where $S_n$ acts on the left by permuting the factors and on the right by relabeling tails of trees.
\item Are compatible with the boundary maps in the sense that for any partition $\sigma$ of $S$ into $S_1$, $S_2$
$$
\phi_\sigma(I_n(x_1, \ldots, x_n)) = \epsilon(\sigma)(I_{S_1^+} \tensor I_{S_2^+}) \left( \bigotimes_{S_1} x_i \tensor \Delta \tensor \bigotimes_{S_2} x_j \right),
$$
where $\Delta$ denotes the Casimir element in $V \tensor V$.
\end{enumerate}

By directly comparing the definitions of spaces $H^\bullet_S$ and curved cyclic $\Lie_\infty$-algebra (\ref{def_ccla}) we arrive at the following statement.
\end{nparagraph}

\begin{proposition}
There is a bijection between the set of cyclic curved $\Lie_\infty$ structures and the set of cohomological field theories on $V$.
\end{proposition}
\qed
Explicitly, let $\{\ell_n\}$ be the collection of $\Lie_\infty$ operations and $Y_n$ are as defined in \ref{def_ccla}. Then we put
$$
I_S(x_1, \ldots, x_n) = \sum_{\tau} Y_\tau(x_1, \ldots, x_n) m(\tau),
$$
where
$$
Y_\tau(x_1, \ldots, x_n) = \left( \bigotimes_{v \in \tau} Y_{|F_v|} \right) \left( \bigotimes_{S} x_i \tensor \Delta^{\tensor E(\tau)} \right).
$$
Here on the right side the Casimir elements $\Delta$ are indexed by the internal edges of $\tau$, and the inputs $x_i$ as well as factors of the Casimir elements are fed
into operations $Y_n$ according to the combinatorics of the tree.

\begin{nparagraph}
\label{par_basis_of_H}
The spaces $H^\bullet_S$ admit a basis who's elements can be represented by certain simple types of trees. Let us first consider the case of the empty set $S$. Then
$H^\bullet_{\ubf{0}}$ is spanned by two elements in degrees $0$ and $1$:
$$
H^\bullet_{\ubf{0}} = \left(\quad
\begin{tikzpicture}
\def\u{2em}
\node (v0) at (0,0) [circle,fill,inner sep=1.5pt] {};
\node (v1) at (2*\u,0) [circle,fill,inner sep=1.5pt] {};
\node (v2) at (3*\u,0) [circle,fill,inner sep=1.5pt] {};
\draw (v1)--(v2);
\end{tikzpicture}
\quad\right).
$$
First of all neither of these elements can be killed by the relations $R(\tau, v, f)$ in (\ref{eq_relations}). Next, consider a tree $\tau$ with
at least $3$ vertices, and assume at first that it contains a part of the form

$$
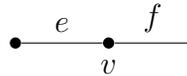
\begin{tikzpicture}
\def\u{3em}
\node (v0) at (0,0) [circle,fill,inner sep=1.5pt] {};
\node (v1) at (1*\u,0) [circle,fill,inner sep=1.5pt,label=below:$v$] {};
\node (v2) at (2*\u,0) {};
\draw (v0)--(v1) node [midway,above] {$e$};
\draw (v1)--(v2) node [midway,above] {$f$};;
\end{tikzpicture}
$$
Then the tree $\tau$ is equal (up to a factor) to $R(\tau', v, f)$, where $\tau'$ is obtained from $\tau$ by contracting edge $e$, and therefore
is killed in $H^\bullet_{\ubf{0}}$.

In general we can always find a vertex $v$ in $\tau$ of the form

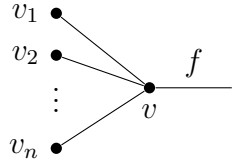
\begin{equation}
\label{eq_dead_ends}
\begin{tikzpicture}[baseline=-0.25em]
\def\u{3em}
\node (v00) at (0,0.8*\u) [circle,fill,inner sep=1.5pt,label=left:$v_1$] {};
\node (v01) at (0,0.35*\u) [circle,fill,inner sep=1.5pt,label=left:$v_2$] {};
\node (v02) at (0,-0.05*\u) {$\vdots$};
\node (v03) at (0,-0.65*\u) [circle,fill,inner sep=1.5pt,label=left:$v_n$] {};
\node (v1) at (1*\u,0) [circle,fill,inner sep=1.5pt,label=below:$v$] {};
\node (v2) at (2*\u,0) {};
\draw (v00)--(v1);
\draw (v01)--(v1);
\draw (v03)--(v1);
\draw (v1)--(v2) node [midway,above] {$f$};
\end{tikzpicture}
\end{equation}
We will show by induction on $n$ that such tree vanishes in $H^\bullet_S$. Let $\tau'$ be the tree obtained from $\tau$ by contracting the edge $(v_n, v)$, then
one of the trees appearing in $R(\tau', v, f)$ will be the tree $\tau$. For the rest of the trees some of the vertices $v_i$ for $1 \le i \le (n - 1)$
are attached instead to $v_n$. In any of these cases vertex $v_n$ will have a smaller number of ``dead ends'' attached to it,
and so by inductive assumption they vanish in $H^\bullet_S$. Hence the tree $\tau$ also vanishes.

Similar considerations show that for $S = \ubf{1}$ we have one-dimensional space $H^\bullet_{\ubf{1}}$ spanned by an element in degree $0$:

$$
H^\bullet_{\ubf{1}} = \left(\quad
\begin{tikzpicture}
\def\u{2em}
\node (v0) at (0,0) [circle,fill,inner sep=1.5pt] {};
\draw (v0)--(v1);
\end{tikzpicture}
\quad\right).
$$

For $S = \ubf{2}$ the space $H^\bullet_{\ubf{2}}$ is infinite dimensional and spanned by trees $\tau_n$ in degree $n$, for $n \ge 0$.

$$
H^\bullet_{\ubf{2}} = \bigoplus_{n \ge 0} \k\tau_n,\quad\quad
\tau_n = 
\begin{tikzpicture}[baseline=-0.25em]
\def\u{2em}
\node (v0) at (0,0) {};
\node (v1) at (1*\u,0) [circle,fill,inner sep=1.5pt,label=below:$v_0$] {};
\node (v12) at (2*\u,0) [circle,fill,inner sep=1.5pt,label=below:$v_1$] {};
\node (v2) at (3*\u,0) {};
\node (v3) at (4*\u,0) {$\cdots$};
\node (v4) at (5*\u,0) {};
\node (v5) at (6*\u,0) [circle,fill,inner sep=1.5pt,label=below:$v_n$] {};
\node (v6) at (7*\u,0) {};
\draw (v0)--(v1)--(v12)--(v2);
\draw (v4)--(v5)--(v6);
\end{tikzpicture}
$$
If the tree contains a vertex, as in (\ref{eq_dead_ends}) then as we saw before, it vanishes in $H^\bullet_S$. Therefore, it is enough to consider
the following situation.

$$
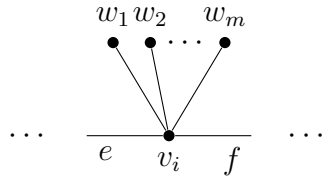
\begin{tikzpicture}[baseline=-0.25em]
\def\u{3em}
\node (v00) at (0.5*\u,0) {$\cdots$};
\node (v0) at (1*\u,0) {};
\node (v1) at (2*\u,0) [circle,fill,inner sep=1.5pt,label=below:$v_i$] {};
\node (v2) at (3*\u,0) {};
\node (v3) at (3.5*\u,0) {$\cdots$};

\node (w1) at (1.4*\u, 1*\u) [circle,fill,inner sep=1.5pt,label=above:$w_1$] {};
\node (w2) at (1.8*\u, 1*\u) [circle,fill,inner sep=1.5pt,label=above:$w_2$] {};
\node (w3) at (2.2*\u, 1*\u) {$\cdots$};
\node (w4) at (2.6*\u, 1*\u) [circle,fill,inner sep=1.5pt,label=above:$w_m$] {};

\draw (v0)--(v1) node [near start,below] {$e$};
\draw (v1)--(v2) node [near end,below] {$f$};
\draw (w1)--(v1);
\draw (w2)--(v1);
\draw (w4)--(v1);
\end{tikzpicture}
$$
We will again argue by induction on $m$. Let $\tau'$ be the tree obtained from $\tau$ by contracting the end $(w_m, v_i)$, then $\tau$ is one
of the trees appearing in $R(\tau', v_i, f)$. As for the other trees we will distinguish two cases.

{\em Case 1.} Some of the vertices $w_i$, for $1 \le i \le (m - 1)$ are attached to $w_m$, while edges $e$ and $f$ remain intact. In this situation
vertex $w_m$ becomes as in (\ref{eq_dead_ends}) and therefore the tree vanishes in $H^\bullet_S$.

{\em Case 2.} The edge $e$ as well as some of $w_i$ are attached to $w_m$. Then the length of the path between tails of the tree (passing through edges $e$ and $f$) is increases by $1$
and both vertices $v_i$ and $w_m$ will have fewer dead ends, so we can proceed by induction on $m$.

In the end we arrive to a tree $\tau_n$, where $n$ is the number of internal edges of the original tree $\tau$.

These arguments can also be applied for general $H^\bullet_S$. Putting all this together we obtain the following statement.
\end{nparagraph}

\begin{proposition} We have the following description of spaces $H^\bullet_S$:
\label{prop_basis_HS}
\begin{enumerate}[label=\alph*)]
\item For $|S| = 0$ we have
$$
H^\bullet_{\ubf{0}} \ \isom\ H^\bullet(\CC P^1).
$$
\item For $|S| = 1$ we have
$$
H^\bullet_{\ubf{1}} \ \isom\ \k.
$$
\item For $|S| = 2$ we have
$$
H^\bullet_{\ubf{2}} \ \isom\ H^\bullet(\CC P^\infty).
$$
\item In general for $|S| \ge 3$ the space $H^\bullet_{\ubf{n}}$ has a basis consisting of metric stable trees, with positive (respectively non-negative) integral lengths
of internal edges (respectively tails).
\end{enumerate}
\end{proposition}
\qed

In fact the first three identifications are isomorphisms of algebras, as we will explain now.

\begin{nparagraph}[Multiplicative structure.]
We will construct a commutative (not super-commutative) product on the spaces $H^\bullet_S$, compatible with the grading, and such that the element $m(\tau_0)$ for the single vertex tree $\tau_0$ with
tails marked by $S$ is the identity of $H^\bullet_S$. Notice, that these conditions already determine the multiplication on $H^\bullet_{\ubf{0}}$ and $H^\bullet_{\ubf{1}}$,
since the square of the degree $1$ element has to vanish. So from now on we will assume that $|S| \ge 2$.

Let $\sigma = \{S_1, S_2\}$ and $\sigma' = \{S'_1, S'_2\}$ be two partitions of $S$, we define a number $a(\sigma, \sigma')$ as the number of distinct non-empty sets among $(S_i \cap S'_j)$ for all $i, j$.
Assume that both partitions are non-trivial (i.e. both $S_1$ and $S_2$ are non-empty), then we have three possibilities:
\begin{enumerate}[label=\arabic*)]
\item $a(\sigma, \sigma') = 2$, which happens if and only if $\sigma = \sigma'$,
\item $a(\sigma, \sigma') = 3$, in which case we will say that partitions $\sigma$ and $\sigma'$ are {\em compatible},
\item $a(\sigma, \sigma') = 4$, in which case we will say that partitions $\sigma$ and $\sigma'$ are {\em incompatible}.
\end{enumerate}

If $e$ is either an internal edge or a tail of a tree $\tau$, then it determines a partition of the set of tails that we will denote by $\sigma(e)$.
According to proposition \ref{prop_basis_HS} it is enough to consider only trees without dead ends, in which case all the partitions $\sigma(e)$ are non-trivial.
Clearly for any two elements $e, f \in E(\tau) \cup \mathrm{Tail}(\tau)$ the number $a(\sigma(e), \sigma(f))$ is either $2$ or $3$.

Let $\sigma$ be a partition of $S$ and define the product $m(\sigma) m(\tau)$ for all $S$-trees $\tau$ as follows.
\begin{enumerate}[label=\alph*)]
\item If $\tau$ has an edge $e$ such that $\sigma$ and $\sigma(e)$ are incompatible, then $m(\sigma)m(\tau) = 0$.
\item If $\sigma$ is compatible with $\sigma(e)$ for all edges and tails then we can construct a tree $\wtilde\tau$ by inserting a new edge $e'$ determined by partition $\sigma$, and
we put $m(\sigma) m(\tau) = m(\wtilde\tau)$.
\item If $\tau$ has an edge or tail $e$ such that $\sigma = \sigma(e)$, then we define the product in one of the two equivalent ways.
\end{enumerate}

{\em Inserting.} Let $e = (v_1, v_2)$ be an edge of $\tau$ (not necessarily unique) with $\sigma(e) = \sigma$, then we construct a tree $\wtilde\tau$ as in case (b), for instance by
adding a new vertex ``in the middle of edge'' $e$ and connecting it only to vertices $v_1$ and $v_2$. Even though such position of the
new inserted vertex is not uniquely determined, the resulting tree doesn't depend on the choice of $e$, since all the edges of $\tau$ defining partition $\sigma$ are arranged
in a single chain.

{\em Transplanting.} Again let us choose an edge $e = (v_1, v_2)$ with $\sigma(e) = \sigma$. Denote by $F_{v_1}^\circ = F_{v_1} - \{e\}$, the set of flags adjacent to $v_1$ except
the one that forms the edge $e$, and similarly $F_{v_2}^\circ = F_{v_2} - \{e\}$. We put
\begin{equation}
\label{eq_transplant}
m(\sigma) m(\tau) = -{1 \over 2} \left( \sum_{F \subset F_{v_1}^\circ \atop |F| \ge 1} m(\tr(\tau, e, F)) + \sum_{F \subset F_{v_2}^\circ \atop |F| \ge 1} m(\tr(\tau, e, F)) \right).
\end{equation}
Here $\tr(\tau, e, F)$ denotes the tree obtained from $\tau$ by inserting a new vertex $v$ in the middle of edge $e$, connecting it to $v_1$ and $v_2$ and then
transplanting flags in $F$ from the original vertex to $v$.

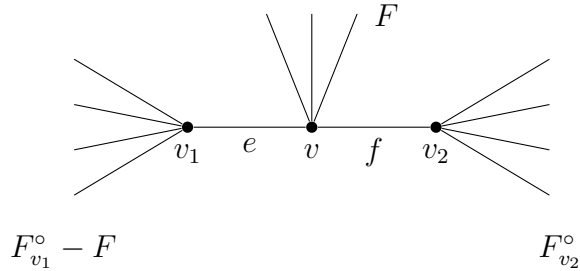
$$
\begin{tikzpicture}[baseline=-0.25em]
\def\u{4em}
\node (v00) at (0,0.6*\u) {};
\node (v01) at (0,0.2*\u) {};
\node (v02) at (0,-0.2*\u) {};
\node (v03) at (0,-0.6*\u) {};
\node (v04) at (0,-1*\u) {$F_{v_1}^\circ - F$};
\node (v1) at (\u,0) [circle,fill,inner sep=1.5pt,label=below:$v_1$] {};
\node (v2) at (2*\u,0) [circle,fill,inner sep=1.5pt,label=below:$v$] {};
\node (v3) at (3*\u,0) [circle,fill,inner sep=1.5pt,label=below:$v_2$] {};
\node (v40) at (4*\u,0.6*\u) {};
\node (v41) at (4*\u,0.2*\u) {};
\node (v42) at (4*\u,-0.2*\u) {};
\node (v43) at (4*\u,-0.6*\u) {};
\node (v44) at (4*\u,-1*\u) {$F_{v_2}^\circ$};

\node (w1) at (1.6*\u, 1*\u) {};
\node (w2) at (2*\u, 1*\u) {};
\node (w3) at (2.4*\u, 1*\u) {};
\node (w4) at (2.6*\u, 0.9*\u) {$F$};

\draw (v00)--(v1);
\draw (v01)--(v1);
\draw (v02)--(v1);
\draw (v03)--(v1);
\draw (v1)--(v2) node [midway,below] {$e$};
\draw (w1)--(v2);
\draw (w2)--(v2);
\draw (w3)--(v2);
\draw (v2)--(v3) node [midway,below] {$f$};
\draw (v3)--(v40);
\draw (v3)--(v41);
\draw (v3)--(v42);
\draw (v3)--(v43);
\end{tikzpicture}
$$

To see that these two descriptions are equivalent, consider the tree $\tau'$ obtained from $\tau$ by contracting edge $e$ to vertex $v$.
Then the relation $R(\tau', v, f)$ contains all the trees from the first sum in (\ref{eq_transplant}) and in addition the tree where
no flags were transplanted to $v$. The latter is exactly the tree described in the inserting procedure. Applying the similar argument to the second sum
we see that both procedures give the same result.

From the description of the product of a partition and a tree it is clear that for any tree $\tau$ we have
$$
m(\tau) = \prod_{e \in E(\tau)} m(\sigma(e)).
$$
Therefore, it extends to the product of two arbitrary trees in $H^\bullet_S$. Moreover, we have the following description of $H^\bullet_S$.

For every non-trivial partition $\sigma$ we consider a symbol $D_\sigma$ of degree $1$ and let $\cF_S = \k[D_\sigma]$ be the free commutative algebra generated by $D_\sigma$ for all
non-trivial partitions $\sigma$.
\end{nparagraph}

\begin{proposition}
For a non-empty set $S$, let $I_S$ by the ideal of $\cF_S$ generated by the following two types of elements.
\begin{enumerate}[label=\alph*)]
\item Products $D_\sigma D_{\sigma'}$ for any two incompatible partitions $\sigma$ and $\sigma'$.
\item Linear combinations over all partitions $\sigma$ of $S^\circ = S - \{s\}$, for any choice of $s \in S$:
$$
\sum_{\sigma = \{S_1, S_2\}} (-1)^{|S_1||S_2|} \epsilon(S_1, S_2) D_{\sigma^+},
$$
where $\sigma^+$ is the partition of $S$ into $S_1$ and $S_2 \sqcup \{s\}$.
\end{enumerate}
Then we have an isomorphism of graded vector spaces $H^\bullet_S \isom \cF_S / I_S$, that sends $D_\sigma$ to $m(\sigma)$.

\end{proposition}
\proof
The proof of this proposition is essentially the repeat of the case by case analysis performed in \cite{Manin} chapter 3, section 4, and replacing usage of Keel's relations
with relations (\ref{eq_relations}).
\qed

\vskip 5em
\section{Weighted graph complex}
\label{sec_graph_com}

The constructions of the previous section can be seen as the tree part (genus zero part) of a version of the graph complex. We begin by recalling relevant notions
to fix the notations.

Consider a graph $G = (V, F, \delta, j)$ (we keep the same notation as in paragraph \ref{par_ribbon_graph}, without the cyclic order of flags around each vertex).
An orientation of $G$ is a generator of the highest exterior power of the free abelian group generated by all internal edges and tails of $G$.
$$
\rmor \in \bigwedge \Z^{E(G) \cup \mathrm{Tail}(G)}.
$$
The purpose of the orientation is to take care of the signs of graphs obtained by contracting or inserting edges.

A weighted modular graph is a graph $G$ equipped with the additional data consisting of $g_v \in \Z_{\ge 0}$ and $w_v \in \Z_{>0}$ for each vertex $v \in V$. We will
refer to $g_v$ as the genus of the vertex $v$, and $w_v$ as its weight. The purpose of numbers $g_v$ is to keep track of the total genus of the graph as we
contract its edges, and the purpose of $w_v$ is to keep track of the multiplicities of vertices. In particular we define the total genus and total weight of a
weighted modular graph as

$$
g(G) = \sum_{v \in V} g_v + \dim H^1(|G|), \quad\quad w(G) = \sum_{v \in V} w_v,
$$
where $|G|$ denotes the geometric realization of $G$.

\begin{nparagraph}
Denote by $C_n$ the free abelian group generated by pairs $(G, \rmor)$, where $G$ is (a representative of) an isomorphism class of connected weighted modular graphs with $n$ tails
and $\rmor$ is its orientation, modulo relations $(G, -\rmor) = -(G, \rmor)$. We consider a grading of $C_n$ by the number of internal edges of $G$ and denote by
$C_n^p$ the corresponding component of degree $p$. Furthermore, it is graded by the total weight and the total genus of $G$ and we will write $C_{g, n, w}^p$ when
it is necessary to specify this additional grading.

Consider a differential on $C_n$ of degree $+1$ defined by expanding a vertex of a graph into a new edge in all possible ways.
$$
d(G, \rmor) = \sum (\wtilde G, e \wedge \rmor),
$$
where the sum is taken over all weighted modular graphs $\wtilde G$ with an edge $e$ connecting some vertices $v$ and $v'$, such that
\begin{enumerate}[label=\alph*)]
\item by removing the edge $e$ and identifying vertices $v$ and $v'$ we obtain graph $G$;
\item let $\bar v$ be the common image of vertices $v$ and $v'$ then $w_{\bar v} = w_v + w_{v'}$ and
$$
g_{\bar v} = \begin{cases}
g_v + g_{v'},\quad\textrm{if } v \neq v'\\
g_v + 1, \quad\textrm{if } v = v';
\end{cases}
$$
\item for all other vertices $u$ of $G$ the weight and genus are the same as in $\wtilde G$.
\end{enumerate}

Clearly this differential is compatible with grading of $C_n^\bullet$ by both weight and genus. Let us restrict to the subcomplex $C^\bullet_{0,n,w}$ of genus zero,
which is spanned by trees, such that genus of every vertex is zero.
\end{nparagraph}

\begin{proposition}
For all $n \ge 1$ we have isomorphisms
$$
H^p_{\ubf{n}} \ \isom\  \bigoplus_{w \ge 1} H^p(C^\bullet_{0,n,w} \tensor \k).
$$
Moreover, for each $w \in \Z_{> 0}$ the cohomology of $C^\bullet_{0,n,w}$ is concentrated in degree $p = w - 1$, and is spanned by trees
with weights of all vertices equal to $1$.
\end{proposition}
\proof
Consider a tree $\tau$ with $p$ vertices of total weight $p + 1$, so that there is exactly one vertex $v$ of weight $2$, while the rest of them are of weight $1$. Notice that
there is at least one flag attached to $v$, due to the assumption that $n \ge 1$ and connectedness of the tree. Applying the differential
of the graph complex $C^\bullet$ to $\tau$ we obtain a linear combination of trees which coincides with the (double of the) relation $R(\tau, v, f)$ in (\ref{eq_relations}) for any choice of flag $f \in F_v$.
Since the number of internal edges of a tree with $p + 1$ vertex is equal to $p$, we conclude that the cohomology $H^p(C^\bullet_{0, n, p+1} \tensor \k) \isom H^p_{\ubf{n}}$.

The vanishing of all the other cohomology groups $H^p(C^\bullet_{0,n,w} \tensor \k)$ for $p < w - 1$ is a consequence of the Koszulity of the operad $\Lie_\infty$.
\qed

\begin{remark}
For $n = 0$ the two-vertex tree spanning $H^1_{\ubf{0}}$ (see paragraph \ref{par_basis_of_H}) is killed in the graph complex by the image of the tree with a single vertex of weight two.
At the level of $\Lie_\infty$-algebra this imposes relation $(c, c) = 0$, where $c$ is the curvature element. In other words this corresponds to the formal isotropic $L$-manifolds.
\end{remark}

\vfill\eject

\end{document}